\documentclass[a4paper]{amsart}
\usepackage{amssymb,bm}
\usepackage{amsmath}
\usepackage{amsfonts}
\usepackage{amsthm}
\usepackage{color}
\usepackage{graphics,graphicx,url}
\usepackage{enumerate}
\usepackage{comment}
\usepackage{centernot}
\usepackage[normalem]{ulem}

\usepackage[utf8]{inputenc}
\usepackage{boondox-cal}

\usepackage{hyperref}
  
  \usepackage{enumerate}

\usepackage{breqn} 

\usepackage[a4paper]{geometry}
\usepackage{tikz}
\usetikzlibrary{calc,decorations.pathmorphing,patterns,arrows,decorations.markings,positioning}
\usetikzlibrary{automata,chains,fit,shapes}

\usepackage{graphicx,pdflscape,rotating,float}
\usepackage{enumitem}

\newtheorem{theorem}{Theorem}%[section]
\newtheorem{lemma}[theorem]{Lemma}

\newtheorem{proposition}[theorem]{Proposition}

\theoremstyle{definition}
\newtheorem{remark}[theorem]{Remark}
\newtheorem{definition}[theorem]{Definition}
\newtheorem{example}[theorem]{Example}

\newcommand{\IndentLength}{\hspace*{1.5em}}
\newcommand{\tab}{}
\newcommand{\tabb}{\IndentLength}
\newcommand{\tabbb}{\tabb\IndentLength}
\newcommand{\tabbbb}{\tabbb\IndentLength}
\newcommand{\tabbbbb}{\tabbbb\IndentLength}
\newcommand{\tabbbbbb}{\tabbbbb\IndentLength}

\usepackage{boondox-cal}
\newcommand{\N}{\mathbb N}

\newcommand{\PTIME}{\ensuremath{\mathsf{PTIME}}}
\newcommand{\PSPACE}{\ensuremath{\mathsf{PSPACE}}}

\newcommand{\NP}{\ensuremath{\mathsf{NP}}}
\newcommand{\PP}{\ensuremath{\mathsf{P}}}
\newcommand{\fclrrs}{finite convergent length-reducing rewriting system}
\newcommand{\icfclrrs}{inverse-closed finite convergent length-reducing rewriting system}
\newcommand{\icf}{icfclrrs}
\newcommand{\nT}{{\mathcal n}_T}
\newcommand{\rT}{{\mathcal r}_T}

\renewcommand{\geq}{\geqslant} \renewcommand{\leq}{\leqslant}  
\newcommand{\Oh}{\mathcal{O}}
\newcommand{\Bset}{\mathcal B}
\newcommand{\Aset}{\mathcal A}
\newcommand{\SLP}{straight-line program}
\newcommand{\SLS}{straight-line sequence}

\newcommand{\gset}{\mathcal S}

\title[The isomorphism problem for plain groups]{The isomorphism problem for plain groups is in $\Sigma_3^\PP$}

\thanks{Research supported by  Australian Research Council grants DP210100271 (Elder, Piggott), DP200100950 (Qiao) DP190100317 (Dietrich). }

 \author[H. Dietrich]{Heiko Dietrich}
\address {Monash University, Clayton VIC 3800, Australia}
 \email{heiko.dietrich@monash.edu}

 \author[M. Elder]{Murray Elder}
 \address{University of Technology Sydney, Ultimo NSW 2007, Australia}  
 \email{murray.elder@uts.edu.au}

 \author[A. Piggott]{Adam Piggott}
 \address{Australian National University, Canberra ACT  0200, Australia}
\email {adam.piggott@anu.edu.au}

 \author[Y. Qiao]{Youming Qiao}
 \address {University of Technology Sydney, Ultimo NSW 2007, Australia}
 \email{youming.qiao@uts.edu.au}

 \author[A. Wei\ss]{Armin Wei{\ss}}
 \address {University of Stuttgart, Universit\"{a}tsstra{\ss}e 53, 70569 Stuttgart, Germany}
 \email{armin.weiss@fmi.uni-stuttgart.de}

\date{\today}



\begin{document}

\maketitle

\begin{abstract}
Testing isomorphism of infinite groups is a classical topic, but from the 
complexity theory viewpoint, few results are known. 
  S{\'e}nizergues and the fifth author  (ICALP2018) proved that the isomorphism problem 
for virtually free groups is decidable in $\mathsf{PSPACE}$ when the input is 
given in terms of so-called {virtually free presentations}. Here we consider the 
isomorphism problem for the class of \emph{plain groups}, that is, 
groups that are isomorphic to a free product of finitely many finite groups and finitely many copies of the infinite cyclic group. 
Every plain group is naturally 
and efficiently presented via an \icfclrrs. 
We prove  that the  
isomorphism problem for plain groups given in this form
 lies in the polynomial time hierarchy, more precisely, in $\Sigma_3^\PP$. 
This result is achieved by combining  new geometric and algebraic characterisations of groups presented by \icfclrrs s developed in recent work of 
the second and third authors (2021) with classical finite group isomorphism results of Babai and Szemer\'edi (1984).
\end{abstract}

\section{Introduction}
\label{sec:intro}

The classical core of combinatorial group theory centres on Dehn's three algorithmic problems concerning finitely presented groups \cite{Dehn}:  given a finite presentation for a group, describe an algorithm that decides whether or not an arbitrary word in the generators and their inverses spells the identity element (the word problem);  given a finite presentation for a group, describe an algorithm that decides whether or not two arbitrary words in the generators and their inverses spell conjugate elements (the conjugacy problem);  describe an algorithm that, given two finite presentations, decides whether or not the  groups presented are isomorphic (the isomorphism problem). For arbitrary finite presentations, these problems are undecidable and so upper bounds on complexity are impossible. 
Obtaining bounds on complexity requires working with presentations that come with a promise that they determine groups in a particular class, or presentations that provide (intrinsic or extrinsic) additional structure.  From Dehn's work, finite length-reducing rewriting systems that satisfy various convergence properties emerged as finite group presentations that simultaneously specify interesting infinite groups and provide natural solutions to the corresponding word problems.  

A research program, active since the 1980s, seeks to classify the groups that may be presented by finite length-reducing rewriting systems satisfying various convergence properties.  The hyperbolic groups \cite{Gromov87}, the virtually-free groups \cite{Hermiller} (groups with a free subgroup of finite index~-- by Muller and Schupp's theorem \cite{MS1} they are also known as context-free groups), and the plain groups \cite{GilmansConjecture} are important classes of groups, each a proper subclass of the class before, that arise within this program.    A group is {\em plain} if it is isomorphic to a free product of finitely many finite groups and a free group of finite rank.  The plain groups may be characterised as the fundamental groups of finite graphs of finite groups with trivial edge groups \cite{KPS1973}, and as the groups admitting a finite group presentation with a simple reduced word problem \cite{HaringSmith}. Moreover, the plain groups are conjectured to be exactly the groups that may be presented by finite convergent length-reducing rewriting systems \cite{MadlenerOttoLengthReducing}.

The isomorphism problem is, of course, the most difficult of Dehn's problems and complexity results concerning this problem are rare.  However, progress has been made on the isomorphism problem for the very classes of groups 
that arise in the study of length-reducing rewriting systems.  Krsti\'{c} solved 
the isomorphism problem for virtually-free groups described by arbitrary finite group presentations 
\cite{Krstic}.  Building on the pioneering work of Rips and Sela \cite{RipsSela}, 
Sela \cite{Sela}, and Dahmani and Groves \cite{DahmaniGroves}, Dahmani and 
Guirardel \cite{DGIso} provided an explicit algorithm that solves the isomorphism 
problem in all hyperbolic groups when the groups are given by finite 
presentations. In light of this result, attention can now shift to complexity bounds for the isomorphism problem. Notice that, in order to obtain complexity bounds, we cannot allow arbitrary presentations as inputs (otherwise we could decide within that complexity bound whether a given presentation is for the trivial group~-- a problem which is undecidable). In \cite{Sen1, Sen2}, S\'{e}nizergues showed that the isomorphism problem for virtually-free groups is primitive recursive when the input is given in the form of two virtually-free presentations, or as two context free grammars.  A virtually-free presentation of a group $G$ specifies a free subgroup 
$F$ plus a set of representatives $S$ for the cosets $F\setminus G$ together with 
relations describing pairwise multiplications of elements from $F$ and $S$; a context-free grammar can specify a virtually-free group by generating the language of words that spell the identity element.  
Then in 2018, S\'{e}nizergues and the {fifth author} \cite{ComplexityIII} showed 
that the isomorphism problem for virtually-free groups can be solved in 
doubly-exponential space when the groups are specified by context-free grammars, and in \PSPACE\ when the groups are given by 
virtually-free presentations.

In the present article, we prove that the complexity bounds for the isomorphism problem in virtually-free groups can be improved significantly when one restricts attention to the class of plain groups. 

 \begin{theorem}[Isomorphism of plain groups]
\label{thm:Sigma3P}
 The isomorphism problem for plain groups presented by 
 inverse-closed finite convergent length-reducing rewriting systems is in $\Sigma_3^\PP$. 
\end{theorem}

We recall that the complexity class $\Sigma_3^\PP$ lies in the {\em polynomial hierarchy}, see for example~\cite[Chapter 5]{AroraBarak}:
\[\NP=\Sigma_1^\PP\subseteq \Sigma_2^\PP\subseteq \Sigma_3^\PP  \dots \subseteq \PSPACE.\] 
Here we use the following specific definition. 
\begin{definition}[\cite{polyh}]
Let $\gset$ be a finite set and $L\subseteq  \gset^*$. 
Then $L\in \Sigma_3^\PP$ if and only if  there is a polynomial $p$ and a predicate $P$ that can be evaluated in \PTIME\
such that 
\[\forall w \in \gset^* \; \left(w\in L \iff \exists x\in \{0,1\}^{p(|w|)}\;  \forall y\in \{0,1\}^{p(|w|)} \; 
\exists z\in \{0,1\}^{p(|w|)} P(w,x,y,z)=1\right).\]
\end{definition}
Rather than specifying the variables as polynomial length binary strings, we will describe data for $x,y,z$ which has polynomial size over a finite alphabet.

\begin{remark} We will abbreviate {\em \icfclrrs} to {\em  \icf} for the rest of this article, and refer to a group admitting a presentation by an \icf\ as an {\em \icf\ group}.

\end{remark}

\begin{remark}In \cite{EP2021} it is shown that the problem of deciding if an \icf\  presents a plain group is in \NP. Note  that if the conjecture that inverse-closed finite convergent rewriting systems can only present plain groups is proved, the word `plain' may be omitted from the statement of Theorem~\ref{thm:Sigma3P}.
\end{remark}

	Note that given an \icf\ for a group, one can compute a context-free grammar for the word problem in polynomial time using the method from \cite{DiekertLengthReducing}. Hence, the results from \cite{ComplexityIII} imply a doubly-exponential-space algorithm for our situation. Therefore, Theorem~\ref{thm:Sigma3P} represents a significant improvement for this special case. In \cite{EP2021} the second and third authors gave a bound of \PSPACE\ for isomorphism of plain groups given as \icf s. This \PSPACE\ algorithm builds upon new geometric and algebraic characterisations of \icf\ groups 
	developed in the same paper.
	Theorem~\ref{thm:Sigma3P} is again a significant improvement on this, lowering the complexity to the third level of the polynomial hierarchy.
The proof of Theorem~\ref{thm:Sigma3P} combines the new characterisations of \icf\ groups from \cite{EP2021}, which enable us to understand  maximal finite subgroup structure and  conjugacy of finite order elements in these groups, with work of Babai and Szemer\'edi \cite{BabaiSz} to test isomorphism of finite groups efficiently using straight-line programs.

Let us briefly give a high-level intuition of the proof of Theorem~\ref{thm:Sigma3P}. Verifying the ranks of the free factors of each group are the same is straightforward, 
 so for this brief description let us   assume the two plain groups are simply free products of $n$ finite groups.  We  existentially guess generating sets $\mathcal{A}_1,\dots, \mathcal{A}_n$  and $\mathcal{B}_1,\dots, \mathcal{B}_n$ for the finite factors in each group such that $|\mathcal{A}_i|=|\mathcal{B}_i|$ and mapping each $\mathcal{A}_i$ to $\mathcal{B}_i$ defines an isomorphism (in other words, we guess an isomorphism defined on generating sets), then, using straight-line programs (and methods from  \cite{BabaiSz}), we universally verify that our guess indeed defines an isomorphism (that $\phi(g)\phi(h) = \phi(gh)$ for all $g,h$). Technically this is an infinite universal branching~-- still, the results of \cite{EP2021} ensure that considering polynomial-length straight-line programs suffice to verify that we guessed an isomorphism correctly.

However, be aware that we also need to verify that the sets we guessed, indeed, generate each group~-- and this is actually the more difficult part. In order to do so, we check that every element of finite order (universal branching) is conjugate to an element in the subgroup generated by some $\mathcal{A}_i$  (existential branching). Again by results in \cite{EP2021} we can restrict to polynomial-length straight-line programs in both the universal and existential branching. Moreover, testing whether an element has finite order can be done in polynomial time. This leads to a $\Sigma_3^\PP$ algorithm.

\subparagraph*{Outline.}
The article is organised as follows. In Section~\ref{sec:prelim}, we provide  background information on rewriting systems, and state the key algebraic results from \cite{EP2021} we need for the present work.  In Section~\ref{sec:SLP}, we review the necessary background on \SLP s for groups. 
In Section~\ref{sec:Monomorphism}, 
we formulate a result which allows us to verify when two finite subgroups of two (potentially infinite) groups are related by an isomorphism, based on a result of  Babai and Szemer\'edi \cite{BabaiSz}.
Section~\ref{sec:main} is devoted to a proof of our main result, Theorem~\ref{thm:Sigma3P}.

\subparagraph*{Notation} Throughout this article we write $\log$ to mean $\log_2$. For $n\in\N_+$ we write $[1,n]$ for the interval $\{1,\dots, n\}\subseteq \N$. If $ \gset$ is an {\em alphabet} (a non-empty finite set), we write $ \gset^\ast$ for the set of finite-length words over $ \gset$, and $|u|$ for the length of the word $u\in  \gset^*$;  the empty word, $\lambda$, is the unique word of length~$0$. For a group $G$, we write $e_G$ for the identity element of $G$.

%%%%%%%%%%%%%%%%%%%%%%%%%%%%%%%%%%%%%%%%%%%%%%%%%%%%%%%%%%%%%%%%%%
\section{Finite convergent length-reducing rewriting systems}\label{sec:prelim}
%%%%%%%%%%%%%%%%%%%%%%%%%%%%%%%%%%%%%%%%%%%%%%%%%%%%%%%%%%%%%%%%%%

\paragraph{Rewriting systems and groups}

Let $\gset$ be a generating set for a group $G$.  If $v,w\in (\gset\cup \gset^{-1})^*$ and $g,h\in G$, then we write $v=_G g$ if the product of letters in $v$ equals  $g$; we write $v=w$ if $v$ and $w$ are identical as words, and $g=h$ if $g$ and $h$ represent the same element of $G$. If $v=_G g$ we say that $v$ {\em spells} $g$.  For example, the identity element $e_G$ is spelled by the empty word $\lambda$, by  $aa^{-1}$ for any $a\in \gset$, and so on. For an integer $r\geq 0$, we define the {\em ball of radius $r$ in $G$ with respect to the generating set $\gset$}, denoted as $B_{e_G}(r)$, to be the set of all elements $g\in G$ for which  there exists a word in  $( \gset\cup \gset^{-1})^*$ of length at most $r$ that spells $g$.   For example, if $G$ is the free abelian group $\langle a,b\mid ab=ba\rangle$ generated by $\gset=\{a,b,a^{-1},b^{-1}\}$ then the ball of radius $2$ is the set of thirteen elements \[\{e_G, a,b,a^{-1},b^{-1}, a^2,ab,b^2,a^{-1}b,a^{-2},a^{-1}b^{-1}, b^{-2},ab^{-1}\}.\]

We briefly recall some basic facts concerning \fclrrs s  necessary for our discussion.  We refer the reader to \cite{BookOtto} for a broader introduction.  
A length-reducing rewriting system is a pair  $( \gset, T)$, where $ \gset$ is a non-empty alphabet, and $T$ is a subset of $ \gset^*\times  \gset^*$, called a set of {\em rewriting rules},
such that for all $(\ell, r) \in T$ we have that $|\ell| > |r|$. We write 
$\rT=\max_{(\ell,r)\in T}\{|r|\}$.

 The set of rewriting rules determines a relation $\to$ on the set $ \gset^\ast$ as follows: $a \to b$ if $a=u\ell v$,  $b=urv$, and $(\ell, r) \in T$.  The reflexive and transitive closure of $\to$ is denoted $\overset{\ast}{\to}$. A word $u \in  \gset^\ast$ is {\em irreducible} if no factor is the left-hand side of any  rewriting rule, and hence $u \overset{\ast}{\to} v$ implies that $u = v$.

The reflexive, transitive and symmetric closure of $\to$ is an equivalence denoted $\overset{\ast}{\leftrightarrow}$.  The operation of concatenation of representatives is well defined on the set of $\overset{\ast}{\leftrightarrow}$-classes, and hence makes a monoid $M = M( \gset, T)$.  We say that $M$ is the {\em monoid presented by $( \gset, T)$}.  When the equivalence class of every letter (and hence also the equivalence class of every word) has an inverse, the monoid $M$ is a group and we say it is {\em the group presented by $( \gset, T)$}.  We note that if a rewriting system $(\gset, T)$ presents a group $G$, then $\langle \gset \mid \ell = r \text{ for } (\ell, r) \in T\rangle$ is a group presentation for $G$.
We say that $( \gset, T)$ (or just $ \gset$) is {\em inverse-closed} if for every $a \in  \gset$, there exists $b \in  \gset$ such that $ab \overset{\ast}{\to}\lambda$.  Clearly, $M$ is a group when $ \gset$ is inverse-closed.

A rewriting system $( \gset, T)$ is {\em finite} if $ \gset$ and $T$ are finite sets, and  {\em terminating} (or {\em noetherian)} if there are no infinite sequences of allowable factor replacements.  It is clear that length-reducing rewriting systems are terminating. A rewriting system is  {\em confluent} if whenever $w\overset{\ast}{\to} x$ and $w\overset{\ast}{\to} y$, there exists $z\in   \gset^*$ such that $x$ and $y$ both reduce to $z$. 
A rewriting system is called {\em convergent} if it is terminating and confluent.  In some literature, finite convergent length-reducing rewriting systems are called finite \emph{Church-Rosser Thue} systems.

We define the {\em size} of a rewriting system $(  \gset, T)$ to be $\nT=|  \gset|+\sum_{(\ell,r)\in T}|\ell r|$, and we note that $\rT\leq \nT$.

\paragraph{Plain groups represented as rewriting systems}
If $G_1, \dots, G_n$ are groups with each $G_i$ presented by  $\langle \gset_i\mid R_i\rangle$  for pairwise disjoint $\gset_1,\dots, \gset_n$, the {\em free product} $G_1\ast \dots \ast G_n$ is the group presented by  $\langle \gset_1\cup\dots \cup \gset_n\mid R_1\cup\dots \cup R_n\rangle$. A group is {\em plain} if it is isomorphic to the free product 
\[  A_1 \ast A_2 \ast \cdots \ast A_p\ast F_r\]
where $p,r$ are non-negative integers, each $A_i$ is a finite group 
and $F_r$ is the free group of rank $r$.

We first observe that every plain group admits a presentation by an \icf\ (see for example {\cite[Corollary 2]{ELDER2022134}}).
\begin{lemma}\label{cor:Plain}
If $G$ is a plain group, then $G$ admits a presentation by a finite convergent length-reducing rewriting system $(\gset, T)$ such that
$\gset=\gset^{-1}$ and the left-hand side of every rule has length $2$.
\end{lemma}

The following fact follows easily from the normal form theory of free 
products (see for example \cite{LyndonSchupp}).
\begin{lemma}\label{lem:PlainIsomorphic}
Two plain groups given as  
\[
  A_1 \ast A_2 \ast \cdots \ast A_p\ast F_r \ \ \text{and}
           \ \  B_1 \ast B_2 \ast \cdots \ast B_q\ast F_s\]
are isomorphic if and only if $p=q, r=s$ and there is a permutation $\sigma$ such that $A_i\cong B_{\sigma(i)}$ for every $i\in[1,p]$. 
\end{lemma}

The following proposition collects key results  about  \icf\ groups proved in 
 \cite{EP2021}. Recall the definitions of $\rT, \nT$ above.

\begin{proposition}
[{\cite[Proposition 15, Lemmas~12,~8,~18]{EP2021}}]\label{prop:EP2021-catch-all}
If $G$ is a plain group presented by an \icf\ $( \gset, T)$, then
\begin{enumerate}
    \item\label{Prop-item1} every finite subgroup $H$ of $G$ is conjugate to a subgroup in $B_{e_G}(\rT+2)$; 
    \item\label{prop-item-conj-classes} the number of conjugacy classes of maximal finite subgroups in $G$ is bounded above by~$\nT^2$; 
    \item\label{Prop-item3} if $g, h \in B_{e_G}(\rT+2)\setminus\{e_G\}$
are conjugate elements of finite order and $t \in G$ is such that $t g t^{-1} =h$, then $t\in B{e_G}(5\rT+4)$;

        \item\label{log_bound}    $\log(|B_{e_G}(\rT+2)|)\leq \nT^2$. 

\end{enumerate}
\end{proposition}

We also make use of the following facts about finite subgroup membership. 
\begin{lemma}
[{\cite[Lemma 11]{EP2021}}]\label{lem:transitivity}
Let  $G$ be a plain group. 
For $g,h\in G$ define $g\sim h$ if $gh$ has finite order. Then 
\begin{enumerate}
    \item\label{transitivity-item}
    the relation $\sim$ 
        is transitive on the set of non-trivial finite-order elements in $G$;
        \item\label{finiteSubgroupCondition-item} a set $\Aset = \{a_1, \dots, a_m\}\subseteq G\setminus \{e_G\}$ generates a finite subgroup if and only if for all $i\in[1,m]$ both $a_i$ and  $a_1a_i$ have finite order;
                \item\label{plainSubgroupMembership-item} if  $A$ is a finite subgroup of $G$ and $g,h\in G$ with $g\in A\setminus \{e_G\}$,  then  $h\in A$ if and only if $h$ and $gh$ have finite order.
         
\end{enumerate}

\end{lemma}

\paragraph{Algorithms for groups in rewriting systems}
Next we observe that deciding if elements have finite order can be done in polynomial time.

\begin{lemma}[Narendran and Otto {\cite[Theorem 4.8]{NO}}]
	\label{lem:finite_order_NarOtto}
	There is a deterministic polynomial-time algorithm for the following problem: given an \icf\ $(\gset, T)$ presenting a group $G$ and a word $u \in \gset^*$, decide whether or not $u$ spells an 
	element of finite order in $G$.
The running time is polynomial in $|T| + |u| + \sum_{(r,\ell)\in T} |\ell|$, so polynomial in $|u|+\nT$.
\end{lemma}

By computing the Smith normal form of a matrix associated to $(\gset, T)$, we have an efficient way to compute the number of infinite cyclic factors of a plain group given by an \icf.

\begin{lemma}
	\label{lem:SmithNormalForm}
		There is a deterministic polynomial-time algorithm for the following problem: 
given an \icf\   $(\gset, T)$ presenting a group $G$,  compute the torsion-free rank of the abelian group $G / [G, G]$. The running time is polynomial in $\nT$, and the torsion-free rank is bounded above by $\nT$.
\end{lemma}
\begin{proof}
Let $G_{\text{ab}}$ denote $G / [G, G]$, the abelianization of $G$. Let $r$ denote the {\em torsion-free rank} of $G_{\text{ab}}$, which is the 
number of  factors  $\mathbb{Z}$ in the free product decomposition of $G$.  We may compute the torsion-free rank of the abelianization $G_{\text{ab}}$ from $(\gset, T)$ in  time that is polynomial in $\nT$ as follows.  Let $ \gset' \subseteq  \gset$ be a subset comprising exactly one generator from each pair of inverses.  The information in $(\gset, T)$ may be recorded in the form of a group presentation $\langle  \gset'  \mid R \rangle$, where
$R$ interprets each rewriting rule in $T$ as a relation over the alphabet $\gset'\cup( \gset')^{-1}$. The information in the presentation 
$\langle  \gset'  \mid R \cup\{[a,b]\mid a,b\in \gset'\}\rangle$ for $G_{ab}$
may be encoded in an $|R| \times | \gset'|$ matrix of integers  $M$.  These integers record the exponent sums of generators in each relation.  The Smith normal form matrix $S$ corresponding to $M$ may be computed in  time that is polynomial in the size of the $|R| \times | \gset'|$ matrix and its entries (see, for example, \cite{PolynomialSNF, SNFComplexityII}), so polynomial in $\nT$. The torsion-free rank of $G_{\text{ab}}$ is the number of zero entries along the diagonal of $S$ (see, for example, \cite[pp. 376-377]{SNF}). Note that this means $r\leq \nT$.
\end{proof}

%%%%%%%%%%%%%%%%%%%%%%%%%%%%%%%%%%%%%%%%%%%%%%%%%%%%%%%%%%%%%%%%%%
\section{Straight-line programs}\label{sec:SLP}
%%%%%%%%%%%%%%%%%%%%%%%%%%%%%%%%%%%%%%%%%%%%%%%%%%%%%%%%%%%%%%%%%%

We use \emph{\SLP{}s} (or more precisely {\em \SLS{}s})
 to represent the elements of a group $A$ with finite generating set 
 $\Aset=\{a_1,\ldots,a_m\}$, see \cite[Section 1.2.3]{Seress} or \cite[Section 
 3]{BabaiSz} for more details; we briefly recall this concept here. Let 
 $X=\{x_1,\ldots,x_m\}$ be a set of abstract symbols of size $m$. 
 A \SLP\ $Y$ of rank $m$ and length $d$ 
on $X$
is a sequence $Y=(s_1,\ldots,s_d)$ where for each $i\in[1,d]$ either  $s_i\in X\cup\{\lambda\}$, or $s_i=s_js_k$ 
 for some $j,k<i$, or $s_i=s_j^{-1}$ for some $j<i$. One says the \SLP\ $Y$ {\em yields} the word $w=s_d\in (X\cup X^{-1})^*$, which we also denote by 
$Y(x_1,\ldots,x_m)=w(x_1,\ldots,x_m)$. We write $Y(a_1,\ldots,a_m)\in (\Aset\cup\Aset^{-1})^*$ for the word that is constructed by replacing every occurrence of 
 $x_i^{\pm 1}$ in $Y(x_1,\ldots,x_m)$ by $a_i^{\pm 1}$. 
 We call the element $g\in A$ such that $Y(a_1,\ldots,a_m)=_Gg$ the {\em evaluation} of $Y$ in $A$ (with respect to $\Aset$).
 
An efficient way to 
store  the \SLP\ is to write instead the operations that define the elements $s_1,\ldots,s_m$ of the sequence (cf.\ \cite[p.\ 10]{Seress}): for example, a generator $s_i=x$ is stored as the pair $(x,+)$, $s_i=\lambda$ is stored as $(\lambda, +)$, an inverse $s_i=s_j^{-1}$ is stored as   $(j,-)$, and a product $s_i=s_js_k$ is stored as $(j,k)$. We call this sequence of operations a {\em \SLS}. The word $Y(a_1,\ldots,a_m)$ can then be computed by following the construction described in this \SLS\ and replacing every generator $x_j$ by $a_j$.  To store this sequence we simply store the address (an integer in $[1,d]$  in binary) and the instruction (an integer  in $[1,m]$ or at most two integers in $[1,d]$  in binary); thus a \SLS\ of rank $m$ and length $d$ requires  $\Oh\left(d(\log(d)+\log(m))\right)$ bits. In what follows, a \SLP\ will always be represented by a \SLS, and we write $Y$ both for a \SLP\ and the \SLS\ representing it.

\begin{example}\label{eg:Z}
Consider the infinite cyclic group  $G$ generated by $\Aset=\{a\}$. The  
\SLS\ $Y=(y_0=(x,+), y_1=(0,0),y_2=(1,1),y_3=(2,2),  y_4=(3,3), y_5= (4,2), y_6= (5,0), y_7= (6,-)) $ 
yields  the word $Y(x)=x^{-21}$ in $(X\cup X^{-1})^*$ with $X=\{x\}$, and $Y(a)$ yields the element $a^{-21}$ of $G$. The \SLS\ $Y=((\lambda,+))$ yields $Y(x)=\lambda$, so $Y(a)=_Ge_G$.
\end{example}

Every element of a finitely generated group with finite generating set $\Aset$ can be  described by a  \SLS: one could first list $\Aset\cup\Aset^{-1}$ using $y_{2i-1}=(x_i,+)$ and $y_{2i}=(2i-1,-)$ for $i\in[1,|\Aset|]$, then choose a word that spells the desired element, and finally construct it letter-by-letter using $y_k=(k-1,j)$ (where $j=2i-1$ if the next letter is $x_i$, and $j=2i$ if the next letter is $x_i^{-1}$). However, Example~\ref{eg:Z} demonstrates that we can sometimes be more efficient than that. In fact, the following result 
shows that elements of a finite group always have {short} \SLS s, with respect to any given generating set.

\begin{lemma}[Babai and Szemer\'edi {\cite[Lemma 7]{BabaiSz}}, Babai \cite{Babai}]
\label{lem:Reachability}
Let $A$ be a finite group with generating set $\Aset=\{a_1,\dots, a_m\}$. For each  $g \in A$, there exists a \SLS\ $Y$ of rank $m$  and of length at most $(\log |A| + 1)^2$ such that $Y(a_1,\dots, a_m)=_Gg$.
\end{lemma}

If $P=(p_{1},\dots, p_{c}),Q=(q_1,\dots, q_d)$ are two \SLS{}s of rank $m$ and length $c,d$ respectively, then we use the notation $[PQ]$ to denote the \SLS\ of rank $m$ and length $c+d+1$ defined as \[[PQ]=(p_1,\dots, p_c,q_1\dots, q_d,(c,c+d)).\]
We call this the {\em product} of $P$ and $Q$, since by construction if $P(x_1,\dots, x_m)=u$ and $Q(x_1,\dots, x_m)=v$ then $[PQ](x_1,\dots, x_m)=uv$.  We denote by  $[PQR]$  the \SLS\ $[[PQ]R]$ of rank $m$ and length $c+d+e+2$ where $R$ has rank $m$ and length $e$.

\paragraph{Compressed word problem} 
We note  that in the setting of groups presented by \icf s, we can efficiently solve the word problem when the input is a \SLS\ representing a group element.

\begin{lemma}[Compressed word problem]
\label{lem:compressedWP}
	There is a deterministic algorithm for the following problem: given an \icf\ $(\gset, T)$ presenting a group $G$, a set
	$\Aset=\{a_1, \dots, a_m\} \subseteq \gset^*$
generating a subgroup $A\leq G$ such that $A\subseteq B_{e_G}(K)$ for some $K \in \mathbb{N}$,  a word $u \in \gset^*$ such that $u=_Gg\in G$, and a \SLS\ $Y$ of rank $m$ and length $d$, decide whether or not $Y(a_1,\dots, a_m)=_A g$.
	
	 The running time is polynomial in $K + \nT +|u|+d+m + \max_i |a_i|$. In particular, if $K$ is bounded by a polynomial in the input size, the algorithm runs in polynomial time. 
\end{lemma}

\begin{proof}
For each $v\in \gset^*$, let $v^{-1}$ denote the formal inverse of $v$ obtained by reversing and replacing each letter $x\in\gset$ by $x^{-1}\in \gset$.

Assume  $Y=(y_1,\dots, y_d)$  where each $y_i=(x_j,+), (j,-)$ or $(j,k)$. 
For $i\in[1,d]$ we compute and store a word $s_i\in \gset^*$ of length at most $K$ as follows:
\begin{itemize}
    \item if $y_i=(x_j,+)$, set $s_i=a_j$;
    \item if $y_i=(j,-)$ with $j<i$, set $s_i=s_j^{-1}$;
    \item if  $y_i=(j,k)$ for $j,k<i$, set $s_i$ to be the reduced word obtained from  $s_js_k$ by applying rewriting rules.
\end{itemize}
Finally, return {\em true} if $s_du^{-1}$ reduces to $\lambda$, and {\em false} otherwise.

Notice that that no $s_i$ becomes longer than $K$. Therefore, each $s_i$ can be computed in time polynomial in $K$ plus the size of the rewriting system and the other data.
\end{proof}

\section{Isomorphism testing for finite subgroups} 
\label{sec:Monomorphism}

In this section we describe an argument based on Babai and Szemer\'edi's work 
\cite{BabaiSz} which we require for proving  Theorem~\ref{thm:Sigma3P}. 
For now our setting is that we are given two groups (later these will be presented 
by rewriting systems) which come with  efficient (polynomial time) algorithms to 
solve the word problem.
Each group will contain some specified finite subgroup, say $A$ in the first group and $B$ in the second. 
We aim to verify in polynomial time the existence of an isomorphism from $A$ to $B$. We start with the following well-known~facts.

\begin{lemma}
\label{lem:finiteGenSet}
Every finite group  $A$ has a  generating set of size  at most $\log|A|$.
\end{lemma}

\begin{proof}
If $\{g_1,\ldots,g_m\}$ is a minimal generating set and $A_n$ is the group generated by $\{g_1,\ldots,g_n\}$ for $n\in[1, m]$, then $|A_1|\geq 2$ and  $|A_{n+1}|\geq 2|A_n|$ for $n\in [1,m-1]$,  so $|A|\geq 2^{m}$ by induction.
\end{proof}

\begin{lemma}
\label{lem:homomEquivalentStatement} Let $A$ and $B$ be groups.
A map $f\colon A\to B$ is a group homomorphism if and only if $f(e_A)=e_B$ and $f(g)f(h)f((gh)^{-1})=e_B$ for all $g,h\in A$
\end{lemma}
\begin{proof} If $f$ is a homomorphism, then these conditions hold. Conversely, the second condition with $g=e_A$ yields $f(h)f(h^{-1})=e_B$, so $f(h^{-1})=f(h)^{-1}$ and $e_B=f(g)f(h)f((gh)^{-1})=f(g)f(h)f(gh)^{-1}$. Thus, $f(g)f(h)=f(gh)$ for all $g,h\in G$.
\end{proof}

We  now state the key technical result, which is the essence of 
\cite[Proposition 4.8]{BabaiSz} where the isomorphism problem for finite groups in the so-called {\em black-box model} is shown to be in $\Sigma_3^\PP$.

\begin{proposition}[Isomorphism between finite subgroups]
	\label{prop:monomorphism}
	Let $A,B$ be finite groups and $K\in \N_+$.	Let  $\Aset=\{a_1,\dots, a_m\}, \Bset=\{b_1,\dots, b_m\}$ be generating sets for $A,B$ respectively, with $m\leq K$.

	Assume that for each $g\in A$ there is a \SLP\ $Y_g$ of rank $m$ and length  at most $K$ 
	such that  $Y_g(a_1,\ldots,a_m)=_A g$, and likewise for $g\in B$ there is a \SLP\ $Z_g$ of rank $m$ and length  at most $K$ such that  $Z_g(b_1,\ldots,b_m)=_B g$.

	Then the map $\psi: \Aset\to\Bset$ with $a_i \mapsto b_i$ induces an isomorphism $A \to B$ if and only if 
	\begin{equation}\label{eqn:IdentityCheck}
	Y(a_1,\ldots,a_m)=_Ae_A \ \Longleftrightarrow \ Y(b_1,\ldots,b_m)=_Be_B\\
	\end{equation}for every \SLP\ $Y$ of rank $m$ and length at most $3K+2$.
\end{proposition}

\begin{proof}

If $\psi$ induces an isomorphism, clearly (\ref{eqn:IdentityCheck}) holds  for all \SLP{}s.

For the converse, 
assume that (\ref{eqn:IdentityCheck}) holds on all  rank-$m$ \SLP s up to length  
$3K+2$.

	Without loss of generality, assume $Y_{e_A}=Z_{e_B}=((\lambda,+))$ and
	\[Y_{a_i}(x_1,\dots, x_m)=Z_{b_i}(x_1,\dots, x_m)=((x_i,+))\]  for each $i\in[1,m]$.
So $Y_{e_A}(a_1,\dots, a_m)=_Ae_A$,  $Z_{e_B}(b_1,\dots, b_m)=_Be_B$, 
$Y_{a_i}(a_1,\dots, a_m)=_Aa_i$ and  $Z_{b_i}(b_1,\dots, b_m)=_Bb_i$ for $i\in[1,m]$.

Define a map $\phi\colon A\to B$ as follows: for $g\in A$, evaluate $Y_g(b_1,\ldots,b_m)$ to get  an element $h\in B$, then 
 set $\phi(g)=h$. 
Thus, $\phi$ maps each $a_i$ to $b_i$.

First, by way of contradiction suppose $\phi$ is not a  homomorphism. Since $\phi(e_A)=e_B$ by definition, 
Lemma~\ref{lem:homomEquivalentStatement} shows that there must exist $g,h\in A$ such 
that $\phi(g)\phi(h)\phi((gh)^{-1})\neq e_B$. This means that in $A$ we have
\[Y_g(a_1,\ldots,a_m)Y_h(a_1,\ldots,a_m)Y_{(gh)^{-1}}(a_1,\ldots,a_m) =_A gh(gh)^{-1} =_A e_A ,\] 
whereas in $B$ we have
\[Y_g(b_1,\ldots,b_n)Y_h(b_1,\ldots,b_n)Y_{(gh)^{-1}}(b_1,\ldots,b_n) =_B \phi(g) \phi(h)\phi((gh)^{-1}) \ne_B e_B.\] 
Let $Y=[Y_gY_hY_{(gh)^{-1}}]$ be the \SLP\ of rank $m$ and length at most $3K+2$. 
Then  $Y$ 
contradicts our  assumption that (\ref{eqn:IdentityCheck}) holds on all  rank-$m$ \SLP s up to length  
$3K+2$.
Thus $\phi$ is a homomorphism. 

Next we show that $\phi$ is injective. If $g\in \ker\phi$, then 
$Y_g(b_1,\dots, 
b_m)$ evaluates to $e_B$; by assumption, (\ref{eqn:IdentityCheck}) holds on 
input $Y_g$, so $g=_A Y_g (a_1,\dots, a_m)=_A e_A$ and  $\phi$ is injective. So we have shown that $\phi$ is a monomorphism which satisfies $\phi:a_i\mapsto b_i$.

Repeating the  preceding  argument for $\phi'\colon B\to A$ defined as: for $g\in B$, evaluate $Z_g(a_1,\ldots,a_m)$ to get  an element $h\in A$, then 
 set $\phi'(g)=h$;  we obtain a monomorphism $\phi'$ with $\phi':b_i\mapsto a_i$. Since $A,B$ are finite this implies that $|A|=|B|$
 hence the monomorphism $\phi$ is an isomorphism, and since $\phi(a_i)=b_i$ for $i\in[1,m]$ we have that  $\phi$ is the (unique) isomorphism induced by $\psi$.\end{proof}

\begin{remark}\label{rmk:Isom} 
Using Lemma~\ref{lem:compressedWP}, we can check condition~(\ref{eqn:IdentityCheck}) in Proposition~\ref{prop:monomorphism} in polynomial time in groups presented by \icf{}s. 

\end{remark}

\section{Proof of the main theorem}\label{sec:main}

The algorithm for the proof of our main theorem checks the conditions of the following proposition. We remark that verifying that some collection of finite subgroups are maximal and that every finite order element is conjugate to an element in  one of these maximal finite subgroups turns out to be the main bottleneck for the  complexity of our algorithm. These are items (2) and (3) in the following proposition.

\begin{proposition}\label{prop:ISOMchecklist}
	Let $G,H$ be plain groups presented by  \icf\ $(\gset,T)$, $(\gset',T')$ respectively. Then $G \cong H$ if and only if
	there are  subgroups $A_i\leq G, B_i\leq H$ for $i \in [1,p]$ such that the following conditions~(\ref{subgroupInBall})--(\ref{TFrank}) are satisfied:
	\begin{enumerate}
    \item\label{subgroupInBall} for each $i \in [1, p]$ we have $A_i\subseteq B_{e_G}(\rT+2)$ and $B_i\subseteq B_{e_H}(\rT'+2)$ (in particular, they are finite subgroups).
	\item\label{subgroupMaximal} each $A_i$ (resp. $B_i$) is a maximal finite subgroup of $G$ (resp. $H$).
\item\label{conjugateInto} every $g \in G\setminus\{e_G\}$ (resp.\ $h \in H\setminus\{e_H\}$) of finite order can be conjugated into exactly one $A_i$ (resp\ $B_i$).
    \item\label{subgroupIso} for each $i \in [1,p]$ we have $A_i\cong B_i$.
	\item\label{TFrank} the torsion-free rank of  $G/[G,G]$ is equal to the torsion-free rank of $H/[H,H]$.
	\end{enumerate}
	
	\noindent	Moreover, we may choose minimal generating sets $\Aset_i\subseteq B_{e_G}(\rT+2)$, $\Bset_i\subseteq B_{e_H}(\rT'+2)$ for $A_i,B_i$ respectively, $i\in[1,p]$ so that for all $i\in[1,p]$:
	\begin{enumerate}\setcounter{enumi}{5}
	\item\label{Aset_size} $|\Aset_i|=|\Bset_i|=m_i\leq \log|A_i|$. 
	\item\label{GensIsom} if  $\Aset_i=\{a_{i,j}\mid j\in[1,m_i]\}$, $\Bset_i=\{b_{i,j}\mid j\in[1,m_i]\}$, the map $a_{i,j}\mapsto b_{i,j}$ for $j\in[1,m_i]$ induces an isomorphism $A_i\to B_i$.
	\end{enumerate}
	
	\noindent Finally, we may replace conditions~(\ref{subgroupInBall})--(\ref{conjugateInto}) by:
	\begin{enumerate}\setcounter{enumi}{7}
	\item\label{MaximalCheckInBall} for every $g \in B_{e_G}(\rT+2)\setminus\{e_G\}$ (resp.\ $h \in B_{e_H}(\rT'+2)\setminus\{e_H\}$) and every $i \in [1, p]$, if $g$ (resp. $h$) and $g a_{i,1}$ (resp. $h b_{i,1}$) have finite order, then $g \in A_i$ (resp. $h \in B_i$). 
	\item\label{conjugateInto_with_t} every $g \in B_{e_G}(\rT+2)\setminus\{e_G\}$ (resp.\ $h \in B_{e_H}(\rT'+2)\setminus\{e_H\}$) of finite order can be conjugated into exactly one $A_i$ (resp\ $B_i$); moreover, $g$ can be conjugated into that $A_i$ (resp\ $B_i$) by a conjugating element of length at most $5\rT+4$ (resp. $5\rT'+4$).
	\item \label{BetterSubgroupContainment} for every $g \in B_{e_G}(\rT+2)\setminus\{e_G\}$ (resp. $h \in B_{e_H}(\rT'+2)\setminus\{e_H\}$), if $g$ (resp. $h$) and $g a_{i,1}$ (resp. $h b_{i,1}$) have finite order, then $ga_{i,j}^{\epsilon} \in B_{e_G}(\rT+2)$ (resp. $h b_{i,j}^{\epsilon} \in B_{e_H}(\rT'+2)$) for every $j \in [1, m_i]$ and $\epsilon \in \{\pm 1\}$.
	
		\end{enumerate}
		
\end{proposition}

\begin{proof}
	First, assume that $G\cong H$.  We will show that conditions~(\ref{subgroupInBall})--(\ref{GensIsom}) are satisfied.  By Lemma~\ref{lem:PlainIsomorphic} there exists an isomorphism  $\psi\colon G\to H$ and free product decompositions
\[	G\cong A_1 \ast A_2 \ast \cdots \ast A_p\ast F_r\ \ \text{and}
	\ \ H\cong B_1 \ast B_2 \ast \cdots \ast B_p\ast F_r\]
	such that $A_i\cong B_i=\psi(A_i)$ for each $i\in[1,p]$. 
	Moreover, by Proposition~\ref{prop:EP2021-catch-all}~(item~\ref{Prop-item1}) we may assume $A_i,B_i$ each lie  within the balls of radius $\rT+2,\rT'+2$ in the Cayley graphs of $(G,\gset),(H,\gset')$ respectively.	
	By Lemma~\ref{lem:finiteGenSet} there exist minimal generating sets $\Aset_i$ and $\Bset_i$ for $A_i,B_i$ for  $i\in[1,p]$ with $|\Aset_i|=|\Bset_i|\leq \log|A_i|$ (we may assume without loss of generality that we choose minimal generating sets to be of the same size). Since $\psi(A_i)=B_i$ we may without loss of generality choose generators so that condition~(\ref{GensIsom}) holds.  The normal form theory for free products (\cite{LyndonSchupp}) gives that: for any $i \neq j$, $A_i \cap A_j = \{e_G\}$ (resp. $B_i \cap B_j = \{e_H\}$); if $p = 0$ then $G$ and $H$ are free groups, and if $p \neq 0$, there are exactly $p$ conjugacy classes of non-trivial maximal finite subgroups in $G$ (resp. $H$) and they are represented by $A_1, \dots, A_p$ (resp. $B_1, \dots, B_p$).  Condition~(\ref{conjugateInto}) follows immediately.  Condition~(\ref{TFrank}) follows immediately from the fact that $G/[G, G] \cong H/[H, H]$.

	 Conversely, suppose there are  subgroups $A_i\leq G, B_i\leq H$ for $i \in [1,p]$ such that conditions~(\ref{subgroupInBall})--(\ref{TFrank})  are satisfied. Conditions~(\ref{subgroupMaximal}) and~(\ref{conjugateInto}) give that every maximal finite subgroup in $G$ (resp. $H$) is conjugate to exactly one of the subgroups $A_1, \dots, A_p$ (resp. $B_1, \dots, B_p)$.  Since $G$ (resp. $H$) is a plain group, it follows that $G \cong A_1 \ast \dots \ast A_p \ast F_r$ (resp. $H = B_1 \ast \dots \ast B_p \ast F_s$) for some free group of rank $r$ (resp. $s$).    Condition (\ref{subgroupIso}) gives that $A_1 \cong B_1, \dots, A_p \cong B_p$.  Condition~(\ref{TFrank}) gives that $r = s$ and $F_r \cong F_s$.  Thus we have that $G \cong H$.

	Now let us show that  conditions~(\ref{subgroupInBall})--(\ref{conjugateInto}) may be replaced by conditions~(\ref{MaximalCheckInBall})--(\ref{BetterSubgroupContainment}).
	First suppose that conditions~(\ref{subgroupInBall})--(\ref{GensIsom})  are satisfied.  Lemma~\ref{lem:transitivity}~(item 3) implies condition~(\ref{MaximalCheckInBall}). Condition~(\ref{conjugateInto}) and Proposition~(\ref{prop:EP2021-catch-all})~(item 3) imply condition~(\ref{conjugateInto_with_t}).  Condition~(\ref{subgroupInBall}) and Lemma~\ref{lem:transitivity}(item 3) imply condition~(\ref{BetterSubgroupContainment}).

	Now suppose that conditions~(\ref{subgroupIso})--(\ref{BetterSubgroupContainment})  are satisfied.  
	To establish that condition~(\ref{BetterSubgroupContainment}) implies condition~(\ref{subgroupInBall}), suppose $A_i$ contains an element $p$ which lies outside $B_{e_G}(\rT+2)$. Let $u\in\Aset^*$ be a word spelling $p$. Then there exists a word $u_1$, an element $a_{i,j}\in\Aset$ and $\epsilon\in\{\pm 1\}$ so that $u_1a_{i,j}^\epsilon$ is a prefix of $u$ such that $u_1$ spells an element that lies in $B_{e_G}(\rT+2)$ and  $u_1a_{i,j}^\epsilon$  spells an element that lies outside $B_{e_G}(\rT+2)$. It follows that $u_1\neq_Ge_G$ (since $|a_{i,j}|\leq \rT+2$). This contradicts condition~(\ref{BetterSubgroupContainment}).  Conditions (\ref{subgroupInBall}) and (\ref{MaximalCheckInBall}) together imply condition~(\ref{subgroupMaximal}).  Condition (\ref{conjugateInto_with_t}) and Proposition~\ref{prop:EP2021-catch-all}~(item~\ref{Prop-item1}) together imply condition~(\ref{conjugateInto}).
	\end{proof}

We are now ready to prove the main result.
\begin{proof}[Proof of Theorem~\ref{thm:Sigma3P}]
We describe a $\Sigma_3^\PP$ algorithm which on input a pair $(\gset,T)$, $(\gset',T')$ of \icf s which are promised to present plain groups, accepts  if and only if the groups are isomorphic. 
Let $N=\max\{\nT,\nT'\}$ be the input size, $G$ the plain group presented by $(\gset,T)$ and $H$ the plain group presented by $(\gset',T')$.

The algorithm needs to demonstrate the existence of some $p\in \mathbb{N}$ and  subgroups $A_i\leq G$ and $B_i\leq H$ for $i\in[1,p]$ which satisfy conditions~(\ref{subgroupIso})--(\ref{BetterSubgroupContainment}) of Proposition~\ref{prop:ISOMchecklist}. We first observe the following. 
By  Proposition~\ref{prop:EP2021-catch-all}~(item~\ref{prop-item-conj-classes}) there are at most $\nT^2$ (resp. $(\nT')^2$) conjugacy classes of maximal finite subgroups in $G$ (resp. $H$), so we have $p\leq N^2$.
By Lemma~\ref{lem:finiteGenSet} and Proposition~\ref{prop:EP2021-catch-all}~(item~\ref{log_bound}), if $\Aset$ (resp. $\Bset$) is a minimal generating set for a maximal finite subgroup $A$ of $G$ (resp. $B$ of $H$), then \[|\Aset|\leq \log|A|\leq \log(|B_{e_G}(\rT+2)|)\leq \nT^2 \leq N^2\] (resp. $|\Bset|\leq (\nT')^2\leq N^2$). By Lemma \ref{lem:Reachability}, for each $g \in A$ (resp. $g \in B$), there exists a \SLS\ $Y$ of length at most $(\log |A|+1)^2 \leq N^4$ (resp. $(\log |B|+1)^2 \leq N^4$) such that $Y$ yields $g$.
Moreover, if $A\cong B$, we may assume they have minimal generating sets of the same size.

We now start with the following quantified statements:
\begin{tabbing}

 \tab   $\exists $ sets $\Aset_i=\left\{a_{i,j}\in  \gset^*\mid  j\in[1,m_i],  |a_{i,j}| \leq \rT+2\right\}$, \\
   \tabb$\Bset_i=\left\{b_{i,j}\in  (\gset')^*,\mid  j\in[1,m_i],  |b_{i,j}| \leq \rT'+2\right\}$\\
    \tabb for $i\in[1,p]$ where $p\leq N^2, m_i\leq N^2$,\\
    \\
    \tabb $\forall$ $(u,v)\in\gset^*\times (\gset')^*$, $|u|\leq \rT+2, |v|\leq \rT'+2$,\\
    \tabbb $(s,s')\in\gset^*\times (\gset')^*$, $|s|\leq 5\rT+4, |s'|\leq 5\rT'+4$,\\
    \tabbb \SLS{}s $Y_i$ of rank $m_i$  and length at most  $3N^4+2$ for each $i\in[1,p]$,\\
    \\
    \tabbb $\exists $ $(t,t')\in\gset^*\times (\gset')^*$, $|t|\leq 5\rT+4, |t'|\leq 5\rT'+4$,\\
    \tabbbb \SLS{}s $Z_1,Z_2$ of rank $m_i$ for some $i\in[1,p]$ and length at most  $N^4$. 
\end{tabbing}

\noindent
Then the  following procedure (predicate) verifies conditions~(\ref{subgroupIso})--(\ref{BetterSubgroupContainment}) in Proposition~\ref{prop:ISOMchecklist} using this data.

First, apply Lemma~\ref{lem:SmithNormalForm} to compute the torsion-free rank of $G/[G,G]$ and $H/[H,H]$ and verify that the rank is the same for both. This establishes condition~(\ref{TFrank}) of Proposition~\ref{prop:ISOMchecklist}.

Next, run  this subroutine:
\begin{tabbing}
      \tabb for $i\in$  $[1,p]$\\
      \tabbb for \= $j\in$ \= $[1,m_i]$\\
      \tabbbb verify that $a_{i,j}$ and $b_{i,j}$ have finite order using Lemma~\ref{lem:finite_order_NarOtto};\\
     \tabbb for \= $j\in [2,m_i]$\\
     \tabbbb verify that $(a_{i,1}a_{i,j})$ and $(b_{i,1}b_{i,j})$ have finite order using Lemma~\ref{lem:finite_order_NarOtto}.
 \end{tabbing}
 This verifies that  $\Aset_i,\Bset_i$ generate finite subgroups by 
 Lemma~\ref{lem:transitivity}~(item~\ref{finiteSubgroupCondition-item}).
 Let $A_i,B_i$ be the names of the  subgroups generated by $\Aset_i,\Bset_i$  respectively.
 We can assume that the algorithm guesses the $\Aset_i,\Bset_i$ to be minimal generating sets of the same size, so we can assume that condition~(\ref{Aset_size}) is satisfied.

Next, we show that the finite subgroups $A_i,B_i$ actually lie inside the ball of radius  $\rT+2$ (resp. $\rT'+2$) by verifying condition~(\ref{BetterSubgroupContainment}) of Proposition~\ref{prop:ISOMchecklist}. Run the following subroutine.
  \begin{tabbing}
 \tabb if $u$  has finite order (using Lemma~\ref{lem:finite_order_NarOtto}) and $u\neq_G e_G$ (reduced word for $u$ is not $\lambda$)\\
  \tabbb    for \= $i\in$ \= $[1,p]$\\
   \tabbbb if \= $(ua_{i,1})$  has finite order (using Lemma~\ref{lem:finite_order_NarOtto}; if so then $u\in A_i$ by  Lemma~\ref{lem:transitivity}~(item~\ref{plainSubgroupMembership-item}))\\
     \tabbbbb      for \= $j\in$ \= $[1,m_i]$\\
     \tabbbbbb compute the reduced word $u_1$ for $ua_{i,j}$ and $u_2$ for  $ua_{i,j}^{-1}$,\\
    \tabbbbbb  verify that $|u_1|,|u_2|\leq \rT+2$.
 \end{tabbing}
Repeat for the word $v$ using the analogous procedure.
This establishes condition~(\ref{BetterSubgroupContainment}).

Next, to verify condition~(\ref{conjugateInto_with_t}) of Proposition~\ref{prop:ISOMchecklist}, we first run this  pre-step.
  \begin{tabbing}
      \tabb for $i\in$ \= $[1,p]$\\
     \tabbb for  $k\in [1,p]\setminus \{i\}$\\
      \tabbbb verify that $(a_{i,1}sa_{k,1}s^{-1})$ and $(b_{i,1}s'b_{k,1}(s')^{-1})$ have infinite order using Lemma~\ref{lem:finite_order_NarOtto}.
 \end{tabbing}
 This shows that no conjugate of $a_{k,1}$ lies in $A_i$  (resp. no conjugate of $b_{k,1}$ lies in $B_i$) for $i\neq k$ by 
  Lemma~\ref{lem:transitivity}~(item~\ref{plainSubgroupMembership-item}) (note that we are running over all $s,s'$ of length at most $5\rT+4, 5\rT'+4$, so all elements in  $ B_{e_G}(5\rT+4), B_{e_H}(5\rT'+4)$).
  
 Now suppose that for some $g\in G\setminus\{e_G\}$ we have $g =_G \alpha^{-1}c\alpha$ and $g =_G \beta d\beta^{-1}$ for some $\alpha,\beta,c,d\in\gset^*$ with $c =_G g_c\in A_i, d =_G g_d\in A_k$ and $i\neq k$. Recall that as in Lemma~\ref{lem:transitivity},  $g\sim h$ means $gh$ has finite order. Then $c=(\alpha\beta)c(\alpha\beta)^{-1}$ so $a_{i,1}\sim (\alpha\beta)d(\alpha\beta)^{-1}$ and $(\alpha\beta)d(\alpha\beta)^{-1}\sim(\alpha\beta)a_{k,1}(\alpha\beta)^{-1}$, so by  Lemma~\ref{lem:transitivity}~(item~\ref{transitivity-item}) $a_{i,1}\sim (\alpha\beta)a_{k,1}(\alpha\beta)^{-1}$ which contradicts the result of the pre-step. It follows that every $g\in G\setminus\{e_G\}$ lies in a conjugate of {\em at most} one subgroup $A_i$. Thus to show condition~(\ref{conjugateInto_with_t}) it suffices to show that every $u\in B_{e_G}(\rT+2)\setminus\{e_G\}$ lies in a conjugate of {\em some} $A_i$ (and analogously for $v$).
 
 We show this with the 
 following subroutine.
  \begin{tabbing}
 	\tabb if $u$  has finite order (using Lemma~\ref{lem:finite_order_NarOtto}) and $u\neq_G e_G$ (reduced word for $u$ is not $\lambda$)\\
 	\tabbb  verify that \= $(uta_{i,1}t^{-1})$  has finite order for some $i\in$ \= $[1,p]$ \\\tabbbb 
 	(using Lemma~\ref{lem:finite_order_NarOtto} with a loop
 	 over all $i\in$ \= $[1,p]$).
 \end{tabbing}
Repeat all of the above for the word $v$ using the analogous procedure (using the word $t'$).
This establishes condition~(\ref{conjugateInto_with_t}) of Proposition~\ref{prop:ISOMchecklist}.

Next, we verify condition~(\ref{MaximalCheckInBall}) of Proposition~\ref{prop:ISOMchecklist}. 
Run this subroutine.
  \begin{tabbing}
  \tabb if $u$  has finite order (using Lemma~\ref{lem:finite_order_NarOtto}) and $u\neq_G e_G$ (reduced word for $u$ is not $\lambda$)\\
  \tabbb     for \= $i\in$ \= $[1,p]$\\
   \tabbbb if \= $(ua_{i,1})$  has finite order (using Lemma~\ref{lem:finite_order_NarOtto})\\
     \tabbbbb verify that $Z_1(a_{i,1}\dots, a_{i,m_i}) =_G u$ using Lemma~\ref{lem:compressedWP}.
 \end{tabbing}
 This shows that if  $u\sim a_{i,1}$ 
 then $g$ can be spelled by a word in $(\Aset_i\cup\Aset_i^{-1})^*$, and so $u=_Gg\in A_i$.
 Repeat for $v$  using the analogous procedure (using the \SLS\ $Z_2$).
 This establishes condition~(\ref{MaximalCheckInBall}) of Proposition~\ref{prop:ISOMchecklist}.

 Lastly, 
to  verify condition~(\ref{GensIsom}) and hence~(\ref{subgroupIso}) of Proposition~\ref{prop:ISOMchecklist},  we check that \[Y_i(a_{i,1},\dots, a_{i,m_i})=e_G \ \Longleftrightarrow \ Y_i(b_{i,1},\dots, b_{i,m_i})=e_H\]
 holds where the $Y_i$ are    \SLS{}s of rank $m_i$  and length at most  $3N^4+2$ for $i\in[1,p]$ which we run through in the universal statement.
 This can be done in polynomial time using   Lemma~\ref{lem:compressedWP}. Then by  Proposition~\ref{prop:monomorphism} (with $K = N^4$), since we are running over all \SLS{}s $Y_i$ of length $3N^4+2$ and rank $m_i$ for all $i\in[1,p]$
we establish condition~(\ref{subgroupIso}).
\end{proof}

\section{Conclusion}

We have shown that the isomorphism problem for plain groups given as \icf s is decidable in $\Sigma_3^\PP$. To the best of our knowledge this presents the smallest complexity bound for the isomorphism problem apart from some very special cases like abelian and free groups (in polynomial time using \cite{PolynomialSNF, SNF, SNFComplexityII}) and finite groups given as Cayley tables (in quasipolynomial time \cite{FN70,Mil78} and also in \NP, and nearly linear time for almost all orders \cite{DWilson}).

There is one obvious open question: can the complexity actually be reduced to $\Sigma_2^\PP$ or even to some smaller class? Note that the obstacle to reach $\Sigma_2^\PP$ is to verify condition (\ref{subgroupMaximal}) and (\ref{conjugateInto}) of Proposition~\ref{prop:ISOMchecklist} via conditions (\ref{MaximalCheckInBall}) and (\ref{conjugateInto_with_t}). 

Another topic for future research is to investigate the maximal size of a finite subgroup presented by an \icf. If one could show a polynomial bound on this size, rather than the exponential bound used here, this could lead to a lower complexity.
However, this question is wide open~-- and, probably, related to the long-standing conjecture that all groups presented by \icf s are plain.

\section*{Acknowledgements}
We wish to thank the reviewers for their helpful comments and corrections.

 \bibliographystyle{plainurl}
 \bibliography{refs}\end{document}